\documentclass[a4paper]{amsart}
\usepackage{graphicx}
\usepackage{amssymb}
\usepackage{amsmath}
\usepackage{amsthm}
\usepackage{amscd}
\usepackage[all,2cell]{xy}

\UseAllTwocells \SilentMatrices
\newtheorem{thm}{Theorem}[section]

\newtheorem{cor}[thm]{Corollary}
\newtheorem{lem}[thm]{Lemma}
\newtheorem{exm}[thm]{Example}

\newtheorem{prop}[thm]{Proposition}
\theoremstyle{definition}

\theoremstyle{remark}

\numberwithin{equation}{section}

\begin{document}
\title[Singular equivalences of trivial extensions]
{Singular equivalences of trivial extensions}

\author[  Xiao-Wu Chen
] {Xiao-Wu Chen}

\thanks{The author is supported by the Fundamental Research Funds for the Central Universities (WK0010000024), Special Foundation of President of The Chinese Academy of Sciences (No.1731112304061) and  National Natural Science Foundation of China (No.10971206).}
\date{\today}

\thanks{E-mail:
xwchen$\symbol{64}$mail.ustc.edu.cn}

\maketitle

\dedicatory{}%
\commby{}%

\begin{abstract}
We prove that a certain pair of bimodules over two artin algebras gives rise to
a triangle equivalence between the singularity categories of the two corresponding
trivial extension algebras. Some consequences and an example are given.
\end{abstract}

\section{Introduction}

Let $k$ be  a commutative artinian ring, and let $A$ be an artin $k$-algebra.  Denote by
$A\mbox{-mod}$ the abelian category of finitely generated left $A$-modules, and by $\mathbf{D}^b(A\mbox{-mod})$
the bounded derived category. Following \cite{Or04}, the \emph{singularity category} $\mathbf{D}_{\rm sg}(A)$
of $A$ is the quotient triangulated category of $\mathbf{D}^b(A\mbox{-mod})$ with respect to the full subcategory
formed by perfect complexes; see also \cite{Buc, KV, Hap91, Ric, Bel2000} and \cite{Kra}. This singularity
category measures the homological singularity of an algebra. For example, an algebra $A$ has finite global dimension
if and only if its singularity category $\mathbf{D}_{\rm sg}(A)$ vanishes. In the meantime, the singularity
 category captures certain stable homological features of the algebra \cite{Buc}.

Two artin algebras $A$ and $B$ are said to be \emph{singularly equivalent} provided that
there is a triangle equivalence between their singularity categories. In this case, the corresponding
equivalence is called a \emph{singular equivalence} between the two algebras. Observe that derived
equivalences induce naturally singular equivalences, while the converse is not true in general. Here, we recall
that a derived equivalence between two algebras is a triangle equivalence between their bounded
derived categories. For examples of singular equivalences, we refer to \cite{Ch09, Ch11, Ch11'}.

The aim of this paper is to construct a new class of singular equivalences, which are induced by a pair of bimodules.
To be more precise, let $A$ and $B$ be two artin $k$-algebras. Let $_AX_B$ and $_BY_A$ be  finitely generated $A$-$B$-bimodule
and $B$-$A$-module, respectively. Here we require that $k$ acts centrally on these bimodules. Then $X\otimes_B Y$ is an $A$-$A$-bimodule. Denote by $\widetilde{A}= A\oplus (X\otimes_B Y)$  the corresponding \emph{trivial extension} algebra; see \cite[p.78]{ARS}. Similarly, we have the trivial extension algebra $\widetilde{B}=B\oplus (Y\otimes_A X)$ of $B$ by the $B$-$B$-bimodule $Y\otimes_A X$.

We have the following result.

\vskip 5pt

\noindent {\bf Theorem A.} \; \emph{Keep the notation as above. Assume that both the algebras
$A$ and $B$ have finite global dimension, and that both the right modules $X_B$ and $Y_A$ are projective. Then there is a singular
equivalence between $\widetilde{A}$ and $\widetilde{B}$, which is induced by the bimodules $X$ and $Y$.}

\vskip 5pt

The precise statement of Theorem A is given in Theorem \ref{thm:main}. In what follows, we describe an immediate
consequence of Theorem A. This consequence, which is essentially due to Smith \cite{Smi'},  is related to the
notion of strong shift equivalence \cite{Will}  in symbolic dynamic system, and also related to some results on
Leavitt path algebras \cite{ALPS}.

By a \emph{modulation pair} we mean a pair $(A, X)$ with $A$ a semisimple artin $k$-algebra and $_AX_A$ a finitely generated $A$-$A$-bimodule on which $k$ acts centrally.  Inspired by the work of Williams \cite{Will} and Smith \cite{Smi'}, we define an \emph{elementary equivalence} between two modulation pairs $(A, X)$ and $(B, Y)$ to be  a pair of bimodules $(_AM_B, {_BN_A})$, on which $k$ acts centrally, such that there are bimodule isomorphisms $X\simeq M\otimes_B N$ and $Y\simeq N\otimes_A M$. Two modulation pairs $(A, X)$ and $(B, Y)$ are \emph{equivalent} provided that there exists a sequence of modulation pairs $(A, X)=(A_1, X_1)$, $(A_2, X_2), \cdots
, (A_n, X_n)=(B, Y)$ such that $(A_i, X_i)$ is elementarily equivalent to $(A_{i+1}, X_{i+1})$ for each $1\leq i\leq n-1$.

A repeated application of Theorem A yields the following result. We point out that in the case that $k$ is a field and the relevant  semisimple algebras are products of copies of $k$, the result is essentially due to Smith, by combining \cite[Theorem 1.2]{Smi'} and \cite[Theorem 7.2]{Smi}.

\vskip 5pt

\noindent {\bf Corollary B}. \; \emph{Let $(A, X)$ and $(B, Y)$ be two modulation pairs which are equivalent. Then there is a singular equivalence between the trivial extension algebras $A\oplus X$ and $B\oplus Y$. \hfill $\square$}

\vskip 5pt

The paper is structured as follows. Section 2 is devoted to recalling some notions on derived categories and
singularity categories. We collect in Section 3 some facts on trivial extension algebras and their singularity categories.
We prove Theorem A in Section 4, where a consequence with an explicit example is included.

For artin algebras, we refer to \cite{ARS}. For derived categories and triangulated categories, we refer to \cite{Ver77} and \cite{Har66}.

\section{Derived categories and singularity categories}

In this section, we recall some notions related to derived categories and
singularity categories of artin algebras.

Let $A$ be an artin algebra over a commutative artinian ring $k$. Recall that
$A\mbox{-mod}$ denotes the category of finitely generated left $A$-modules. We denote by
$A\mbox{-proj}$ the full subcategory formed by projective modules.

A complex $X^\bullet=(X^n, d_X^n)$ of $A$-modules consists of a sequence $X^n$ of $A$-modules
together with differentials $d_X^n\colon X^n\rightarrow X^{n+1}$ subject to the relations $d_X^{n+1}\circ d_X^n=0$.
Denote by $X^\bullet[1]$ the \emph{shifted complex} of $X^\bullet$, which is given by $(X^\bullet[1])^n=X^{n+1}$ and $d^n_{X^\bullet[1]}=-d_X^{n+1}$. This gives  rise to the shift functor [1] on the category of complexes;
it is an automorphism. For a chain map $f^\bullet\colon  X^\bullet \rightarrow Y^\bullet$ between complexes of $A$-modules, its \emph{mapping cone} ${\rm Con}(f^\bullet)$ is a complex defined by ${\rm Con}(f^\bullet)^n=X^{n+1}\oplus Y^n$ and $d_{{\rm Con}(f^\bullet)}^n=\begin{pmatrix} -d_X^{n+1} & 0\\
 f^{n+1} & d_Y^n \end{pmatrix}$.  Then there exists a chain map, which is called  the \emph{natural projection}, $p^\bullet\colon {\rm Con}(f^\bullet)\rightarrow X^\bullet[1]$ such that $p^n=({\rm Id}_{X^{n+1}}, 0)$.

A complex $X^\bullet$ is bounded provided that only finitely many $X^n$'s are nonzero. Recall that $\mathbf{D}^b(A\mbox{-mod})$ denotes the bounded derived category of $A\mbox{-mod}$, whose shift functor is also  denoted by $[1]$.  The module category $A\mbox{-mod}$ is viewed as a full subcategory of  $\mathbf{D}^b(A\mbox{-mod})$ by identifying an $A$-module with the corresponding stalk complex concentrated at degree zero (\cite[Proposition I.4.3]{Har66}).

 Recall that short exact sequences of complexes induce triangles in derived categories. For this, let  $0\rightarrow X^\bullet
 \stackrel{f^\bullet}\rightarrow Y^\bullet \stackrel{g^\bullet}\rightarrow Z^\bullet \rightarrow 0$ be a short exact sequence of bounded complexes of $A$-modules. Then the chain map $t^\bullet\colon {\rm Con}(f^\bullet)\rightarrow Z^\bullet$ defined as
 $t^n=(0, g^n)$ is a quasi-isomorphism. In particular, $t^\bullet$ is invertible in $\mathbf{D}^b(A\mbox{-mod})$. Then we have the following induced triangle in $\mathbf{D}^b(A\mbox{-mod})$
 \begin{align}
 X^\bullet \stackrel{f^\bullet}\longrightarrow Y^\bullet \stackrel{g^\bullet}\longrightarrow Z^\bullet\stackrel{p^\bullet\circ (t^\bullet)^{-1}}\longrightarrow X^\bullet[1].
 \end{align}
 For details, we refer to \cite[Proposition I.6.1]{Har66} and the remark thereafter.

Recall that a complex in $\mathbf{D}^b(A\mbox{-mod})$  is \emph{perfect} provided that it is isomorphic to a bounded complex consisting of projective modules; these complexes form a full triangulated subcategory ${\rm perf}(A)$. Recall that, via an obvious functor,  ${\rm perf}(A)$ is triangle equivalent to the bounded homotopy category $\mathbf{K}^b(A\mbox{-proj})$; compare \cite[1.1-1.2]{Buc}. As a consequence,  an $A$-module, viewed as a stalk complex in $\mathbf{D}^b(A\mbox{-mod})$, is perfect
if and only if it has finite projective dimension.

Following \cite{Or04}, we call the quotient triangulated category $$\mathbf{D}_{\rm sg}(A)=\mathbf{D}^b(A\mbox{-mod})/{{\rm perf}(A)}$$ the \emph{singularity category} of $A$. Denote by $q\colon \mathbf{D}^b(A\mbox{-mod})\rightarrow \mathbf{D}_{\rm sg}(A)$ the quotient functor. We denote the shift functor on $\mathbf{D}_{\rm sg}(A)$ also by $[1]$, whose inverse is denoted by $[-1]$. Recall from \cite[Chaptre 1, \S 2]{Ver77} that for a triangle $X^\bullet \rightarrow Y^\bullet \rightarrow Z^\bullet \stackrel{a} \rightarrow X^\bullet[1]$ in $\mathbf{D}^b(A\mbox{-mod})$ with $Y^\bullet$ perfect, we have that $q(a)$ is an isomorphism in $\mathbf{D}_{\rm sg}(A)$.

\section{Trivial extensions}

In this section, we recall some facts on modules over trivial extension algebras and study their singularity catgeories.

Let $A$ be an artin $k$-algebra. Let $_AX_A$ be a finitely generated $A$-$A$-bimodule, on which $k$ acts
centrally. The corresponding \emph{trivial extension} algebra $T=A\oplus X$ has its multiplication given by
$(a, x)(a', x')=(aa', a.x'+x.a')$; see \cite[p.6]{FGR} and \cite[p.78]{ARS}. Here, we use ``." to denote the $A$-actions on $X$. Then $T$ is also an artin $k$-algebra, and $A$ is naturally viewed as a subalgebra of $T$.

We consider the category $T\mbox{-mod}$ of finitely generated left $T$-modules.   We identify a left $T$-module
with a pair $(M, \sigma)$, where $M$ is a left $A$-module and $\sigma\colon X\otimes_A M\rightarrow M$ is a morphism
of left $A$-modules with the property $\sigma\circ ({\rm Id}_X\otimes \sigma)=0$; see \cite[Section 1]{FGR}. Then a morphism $(M, \sigma)\rightarrow (N, \delta)$ of $T$-modules is just a morphism $f\colon M\rightarrow N$ of $A$-modules satisfying $f\circ \sigma=\delta\circ ({\rm Id}_X\otimes f)$. We write $f\colon (M, \sigma)\rightarrow (N, \delta)$. Observe that the regular $T$-module $_TT$ is identified with the pair $(A\oplus X, \begin{pmatrix} 0 & 0 \\ {\rm Id}_X & 0\end{pmatrix})$. Here, we identify $X\otimes_A (A\oplus X)$ with $X\oplus (X\otimes_A X)$.

Consider the functor $T\otimes_A -\colon A\mbox{-mod}\rightarrow T\mbox{-mod}$. In view of the above identification, we have for an $A$-module $L$, an identification of left $T$-modules
\begin{align}\label{equ:1}
T\otimes_A L=(L\oplus (X\otimes_A L), \begin{pmatrix} 0 & 0\\ {\rm Id}_{X\otimes_A L} & 0 \end{pmatrix}).
\end{align}
Here, $L\oplus (X\otimes_A L)$ is viewed as an $A$-module, and $\begin{pmatrix} 0 & 0\\ {\rm Id}_{X\otimes_A L} & 0 \end{pmatrix}\colon X\otimes_A(L\oplus (X\otimes_A L))\rightarrow L\oplus (X\otimes_A L)$ is well defined, since we
identify $X\otimes_A(L\oplus (X\otimes_A L))$ with $(X\otimes_A L)\oplus (X\otimes_A X\otimes_A L)$.

We introduce two endofunctors on $T\mbox{-mod}$.  Define a functor $S\colon T\mbox{-mod}\rightarrow T\mbox{-mod}$ such that $S((M, \sigma))=(M, -\sigma)$ and $S(f)=f$. This is an automorphism; moreover, we have $S^2={\rm Id}_{T\mbox{{\small -mod}}}$.  Observe the isomorphism $\begin{pmatrix} {\rm Id}_A & 0 \\ 0 & -{\rm Id}_X\end{pmatrix} \colon S(T)\simeq T$ of $T$-modules. Another functor $X\otimes_A- \colon T\mbox{-mod}\rightarrow T\mbox{-mod}$ is given such that it sends  $(M, \sigma)$ to $(X\otimes_A M, {\rm Id}_X\otimes\sigma)$, and sends $f$ to ${\rm Id}_X\otimes f$.

The following observation is quite useful.

\begin{lem}
Keep the notation above. Then  for any $T$-module $(M, \sigma)$, we have  the following exact sequence of
$T$-modules
\begin{align}\label{equ:se1}
0\longrightarrow (X\otimes_A M, -{\rm Id}_X\otimes \sigma)\stackrel{\begin{pmatrix} -\sigma\\ {\rm Id}_{X\otimes_A M} \end{pmatrix}}
\longrightarrow T\otimes_A M \stackrel{({\rm Id}_M, \sigma)}\longrightarrow (M, \sigma)\longrightarrow 0.
\end{align}
Moreover, the sequence is functorial, that is, it is natural in the $T$-module $(M, \sigma)$.
\end{lem}

We consider the \emph{restriction functor} ${\rm res}\colon T\mbox{{\rm -} {\rm mod}}\rightarrow A\mbox{-mod}$, which sends
$(M, \sigma)$ to $M$. Then the exact sequence (\ref{equ:se1}) gives rise to the following exact sequence of endofunctors on
$T\mbox{-mod}$
\begin{align}
0\longrightarrow (X\otimes_A -)\circ S \longrightarrow (T\otimes_A -)\circ {\rm res} \longrightarrow {\rm Id}_{T\mbox{-mod}}\longrightarrow 0.
\end{align}

\vskip 5pt

\begin{proof}
We use the identification (\ref{equ:1}) of $T$-modules. Then the proof is done by direct verification.
\end{proof}

We consider bounded complexes of $T$-modules. As for modules, a complex of $T$-modules is identified with a
pair $(M^\bullet, \sigma^\bullet)$, where $M^\bullet=(M^n, d_M^n)_{n\in \mathbb{Z}}$ is a complex of $A$-modules and
$\sigma^\bullet \colon X\otimes_A M^\bullet\rightarrow M^\bullet$ is a chain map between  complexes of $A$-modules satisfying
$\sigma^\bullet \circ ({\rm Id}_X\otimes \sigma^\bullet)=0$.

Recall that  the exact sequence (\ref{equ:se1}) is natural and then it extends to complexes. More precisely, for a complex $(M^\bullet, \sigma^\bullet)$ of $T$-modules, we have an exact sequence of complexes of $T$-modules
\begin{align}\label{equ:2}
0\longrightarrow (X\otimes_A M^\bullet, -{\rm Id}_X\otimes \sigma^\bullet)\stackrel{\begin{pmatrix} -\sigma^\bullet\\ {\rm Id}_{X\otimes_A M^\bullet} \end{pmatrix}}
\longrightarrow T\otimes_A M^\bullet \stackrel{({\rm Id}_{M^\bullet}, \sigma^\bullet)}\longrightarrow (M^\bullet, \sigma^\bullet)\longrightarrow 0.
\end{align}

\vskip 5pt

We consider the bounded derived category  $\mathbf{D}^b(T\mbox{-mod})$ of $T$-modules. We assume that the right $A$-module $X_A$ is projective. Then the exact endofunctor $(X\otimes_A -)\circ S$ on $T\mbox{-mod}$ extends naturally to a triangle endofunctor on $\mathbf{D}^b(T\mbox{-mod})$. The triangle endofunctor is still denoted by $(X\otimes_A -)\circ S$.

We will define a natural transformation
$$\eta\colon {\rm Id}_{\mathbf{D}^b(T{\rm \mbox{\small -mod}})}\longrightarrow [1]\circ (X\otimes_A -)\circ S$$
 of triangle  endofunctors on $\mathbf{D}^b(T\mbox{-mod})$. For this, recall that for a complex $(M^\bullet, \sigma^\bullet)$ of $T$-modules, the exact sequence (\ref{equ:2}) gives up to a quasi-isomorphism
$$t^\bullet\colon {\rm Con}(\begin{pmatrix} -\sigma^\bullet\\ {\rm Id}_{X\otimes_A M^\bullet} \end{pmatrix})\longrightarrow  (M^\bullet, \sigma^\bullet).$$
In particular, $t^\bullet$ is invertible in $\mathbf{D}^b(T\mbox{-mod})$. Denote by $p^\bullet\colon {\rm Con}(\begin{pmatrix} -\sigma^\bullet\\ {\rm Id}_{X\otimes_A M^\bullet} \end{pmatrix})\rightarrow (X\otimes_A M^\bullet, -{\rm Id}_X\otimes \sigma^\bullet)[1]$ the natural projection.  Set $\eta_{(M^\bullet, \sigma^\bullet)}=p^\bullet \circ (t^\bullet)^{-1}$.
Then the short exact sequence (\ref{equ:2}) induces the following triangle in $\mathbf{D}^b(T\mbox{-mod})$
\begin{align}\label{equ:3}
(X\otimes_A M^\bullet, -{\rm Id}_X\otimes \sigma^\bullet)
\rightarrow T\otimes_A M^\bullet \rightarrow (M^\bullet, \sigma^\bullet)\stackrel{\eta_{(M^\bullet, \sigma^\bullet)}}\longrightarrow
(X\otimes_A M^\bullet, -{\rm Id}_X\otimes \sigma^\bullet)[1].
\end{align}
For details on the construction of $t^\bullet$ and $p^\bullet$, we refer to Section 2.

\begin{prop}\label{prop:1}
Keep the notation as above. Assume that $X_A$ is projective. Then $\eta\colon {\rm Id}_{\mathbf{D}^b(T{\rm \mbox{\small -{\rm mod}}})}\rightarrow [1]\circ (X\otimes_A -)\circ S$ is a natural transformation of triangle endofunctors on $\mathbf{D}^b(T\mbox{-{\rm mod}})$.
\end{prop}

\begin{proof}
Observe that both the two chain maps $t^\bullet$ and $p^\bullet$ are functorial in $(M^\bullet, \sigma^\bullet)$  as objects in the category of complexes of $T$-modules.  Recall that a morphism $(M^\bullet, \sigma^\bullet) \rightarrow (M'^\bullet, \sigma'^\bullet)$ in $\mathbf{D}^b(T\mbox{-mod})$ is represented as $(M^\bullet, \sigma^\bullet)\stackrel{(s^\bullet)^{-1}}\rightarrow (N^\bullet, \delta^\bullet)\stackrel{a^\bullet}\rightarrow (M'^\bullet, \sigma'^\bullet)$ such that $s^\bullet$ and $a^\bullet$ are some chain maps, and $s^\bullet$ is a quasi-isomorphism. Then it follows that both $t^\bullet$ and $p^\bullet$ are functorial in $(M^\bullet, \sigma^\bullet)$ as objects in the bounded derived category $\mathbf{D}^b(T\mbox{-mod})$. From this, we infer that $\eta$ is a natural transformation.
\end{proof}

We observe that  the automorphism $S\colon \mathbf{D}^b(T\mbox{-mod})\rightarrow \mathbf{D}^b(T\mbox{-mod})$, induced from the automorphism $S$ on $T\mbox{-mod}$, sends perfect complexes to perfect complexes, since $S(T) \simeq T$. Then we have the induced functor $S\colon \mathbf{D}_{\rm sg}(T)\rightarrow \mathbf{D}_{\rm sg}(T)$, which is also an automorphism.

We assume further that the algebra $A$ has finite global dimension. Then each bounded complex of $A$-modules
is perfect. Consider the triangle  (\ref{equ:3}) for any complex $(M^\bullet, \sigma^\bullet)$ of $T$-modules. Then the complex
$T\otimes_A M^\bullet$ is perfect. Recall that ${\rm perf}(T)$ is a triangulated subcategory of $\mathbf{D}^b(T\mbox{-mod})$. Then the triangle (\ref{equ:3}) implies that $(M^\bullet, \sigma^\bullet)$ is perfect if and only if $(X\otimes_A M^\bullet, -{\rm Id}_X\otimes \sigma^\bullet)$ is perfect.  From this, we infer that the functor $X\otimes_A- \colon \mathbf{D}^b(T\mbox{-mod}) \rightarrow \mathbf{D}^b(T\mbox{-mod})$ sends perfect complexes to perfect complexes. Then it induces
the corresponding triangle endofunctor on $\mathbf{D}_{\rm sg}(T)$, which is  denoted by $X\otimes_A-\colon \mathbf{D}_{\rm sg}(T) \rightarrow \mathbf{D}_{\rm sg}(T)$.

\begin{prop}
Keep the notation as above. Assume that $A$ has finite global dimension and $X_A$ is projective. Then there is a natural isomorphism $${\rm Id}_{\mathbf{D}_{\rm sg}(T)}\simeq [1]\circ (X\otimes_A -)\circ S$$
 of triangle endofunctors on $\mathbf{D}_{\rm sg}(T)$. In particular, the triangle functor
 $$X\otimes_A-\colon \mathbf{D}_{\rm sg}(T) \longrightarrow \mathbf{D}_{\rm sg}(T)$$ is an auto-equivalence.
\end{prop}

\begin{proof}
Observe that the second statement is an immediate consequence of the first one, since both $[1]$ and $S$ are automorphisms on
$\mathbf{D}_{\rm sg}(T)$.

 For the first statement, recall the natural transformation  $\eta\colon {\rm Id}_{\mathbf{D}^b(T{\rm \mbox{\small -{\rm mod}}})}\rightarrow [1]\circ (X\otimes_A -)\circ S$ in Proposition \ref{prop:1}. It induces the natural transformation  $\eta\colon {\rm Id}_{\mathbf{D}_{\rm sg}(T)}\rightarrow [1]\circ (X\otimes_A -)\circ S$ between endofunctors on $\mathbf{D}_{\rm sg}(T)$. Fix any bounded complex $(M^\bullet, \sigma^\bullet)$ of $T$-modules. By the assumption, the complex $M^\bullet$ of $A$-modules is perfect. Hence, the complex $T\otimes_A M^\bullet$ of $T$-modules is perfect.  Consider the triangle (\ref{equ:3}) in $\mathbf{D}^b(T\mbox{-mod})$. Then  we deduce that $\eta_{(M^\bullet, \sigma^\bullet)}$ is an isomorphism in $\mathbf{D}_{\rm sg}(T)$; see Section 2.  This completes the proof.
\end{proof}

\section{Proof of Theorem A}

We prove Theorem A and discuss a consequence with an explicit example in this section.

Let $A$ and $B$ be two artin $k$-algebras. Let $_AX_B$ and $_BY_A$ be two finitely generated bimodules, on which
$k$ acts centrally. Then $X\otimes_B Y$ and $Y\otimes_A X$ are naturally an $A$-$A$-bimodule and $B$-$B$-bimodule, respectively.
Consider the trivial extension algebras $\widetilde{A}=A\oplus (X\otimes_B Y)$ and $\widetilde{B}=B\oplus (Y\otimes_A X)$.

Recall from Section 3 that a left $\widetilde{A}$-module is identified with a pair $(M, \sigma)$ such that $M$ is a left $A$-module and $\sigma\colon (X\otimes_B Y)\otimes_A M\rightarrow M$ is a morphism of $A$-modules satisfying $\sigma\circ ({\rm Id}_{X\otimes_B Y}\otimes \sigma)=0$.

Consider the functor $Y\otimes_A -\colon \widetilde{A}\mbox{-mod}\rightarrow \widetilde{B}\mbox{-mod}$, which sends an
$\widetilde{A}$-module $(M, \sigma)$ to $(Y\otimes_A M, {\rm Id}_Y\otimes \sigma)$,  and a morphism $f$ to ${\rm Id}_Y\otimes f$.
Here, to view $(Y\otimes_A M,  {\rm Id}_Y\otimes \sigma)$ as a $\widetilde{B}$-module, we identify $(Y\otimes_A X)\otimes_B (Y\otimes_A M)$ with $Y\otimes_A ((X\otimes_B Y)\otimes_A M)$.

We observe the following result. Recall that $B$ is viewed as  a subalgebra of $\widetilde{B}$. Then for a left $B$-module $L$, $\widetilde{B}\otimes_B L$ is a left $\widetilde{B}$-module.

\begin{lem}\label{lem:iso}
Keep the notation as above. Then there is an isomorphism of $\widetilde{B}$-modules
$$Y\otimes_A \widetilde{A}\simeq \widetilde{B}\otimes_B Y.$$
\end{lem}

\begin{proof}
We identify the regular $\widetilde{A}$-module $\widetilde{A}$ with the pair $(A\oplus (X\otimes_B Y), \begin{pmatrix} 0 & 0 \\
{\rm Id}_{X\otimes_B Y} & 0 \end{pmatrix})$. Then we have  $Y\otimes_A \widetilde{A}= (Y\oplus (Y\otimes_A X\otimes_B Y), \begin{pmatrix} 0 & 0 \\
{\rm Id}_{Y\otimes_A X\otimes_B Y} & 0 \end{pmatrix})$. Here, we use the natural isomorphism $Y\otimes_A A
\simeq Y$.  Applying the identification (\ref{equ:1}) to $\widetilde{B}\otimes_B Y$, we are done.
\end{proof}

Assume that the right modules $X_B$ and $Y_A$ are projective, and that the left $B$-module $_BY$ has finite projective dimension.
In this case, $\widetilde{B}$ viewed as a right $B$-module  is projective. It follows that the left $\widetilde{B}$-module $\widetilde{B}\otimes_B Y$ has finite projective dimension, and by Lemma \ref{lem:iso} so does $Y\otimes_A\widetilde{A}$. Then the exact functor $Y\otimes_A- \colon \widetilde{A}\mbox{-mod}\rightarrow \widetilde{B}\mbox{-mod}$ extends to the corresponding triangle functor $Y\otimes_A- \colon \mathbf{D}^b(\widetilde{A}\mbox{-mod})\rightarrow \mathbf{D}^b(\widetilde{B}\mbox{-mod})$. The triangle functor sends perfect complexes to perfect complexes, since the $\widetilde{B}$-module $Y\otimes_A \widetilde{A}$ has finite projective dimension, that is, it is perfect as a stalk complex. Then we have the induced triangle functor $Y\otimes_A- \colon \mathbf{D}_{\rm sg}(\widetilde{A})\rightarrow \mathbf{D}_{\rm sg}(\widetilde{B})$.

Similar as the above, we have the functor $X\otimes_B -\colon \widetilde{B}\mbox{-mod}\rightarrow \widetilde{A}\mbox{-mod}$, which satisfies that  $X\otimes_B \widetilde{B}\simeq \widetilde{A}\otimes_A X$. If we assume that  right modules $X_B$ and $Y_A$ are projective, and that the left $A$-module $_AX$ has finite projective dimension, we have the naturally induced functor $X\otimes_B- \colon \mathbf{D}_{\rm sg}(\widetilde{B})\rightarrow \mathbf{D}_{\rm sg}(\widetilde{A})$.

We have our main result, which is Theorem A in the introduction. Recall from Section 3 the automorphism $S_A\colon \mathbf{D}_{\rm sg}(\widetilde{A})\rightarrow \mathbf{D}_{\rm sg}(\widetilde{A})$, which sends a complex $(M^\bullet, \sigma^\bullet)$ to $(M^\bullet, -\sigma^\bullet)$. Similarly, we have the automorphism $S_B\colon \mathbf{D}_{\rm sg}(\widetilde{B})\rightarrow \mathbf{D}_{\rm sg}(\widetilde{B})$.

\begin{thm}\label{thm:main}
Let $A$ and $B$ be artin algebras of finite global dimension, and let $_AX_B$ and $_BY_A$ be bimodules such that  both the right modules $X_B$ and $Y_A$ are projective. Let $\widetilde{A}$ and $\widetilde{B}$ be as above. Then the triangle  functor
$$Y\otimes_A- \colon \mathbf{D}_{\rm sg}(\widetilde{A})\longrightarrow \mathbf{D}_{\rm sg}(\widetilde{B})$$
is an equivalence, whose quasi-inverse is given by
$$\mathbf{D}_{\rm sg}(\widetilde{B}) \stackrel{X\otimes_B-}\longrightarrow \mathbf{D}_{\rm sg}(\widetilde{A}) \stackrel{[-1]}\longrightarrow \mathbf{D}_{\rm sg}(\widetilde{A})\stackrel{S_A}\longrightarrow \mathbf{D}_{\rm sg}(\widetilde{A}).$$
\end{thm}

\begin{proof}
Consider the composite functor $\mathbf{D}_{\rm sg}(\widetilde{A}) \stackrel{Y\otimes_A-}\longrightarrow  \mathbf{D}_{\rm sg}(\widetilde{B}) \stackrel{X\otimes_B-}\longrightarrow \mathbf{D}_{\rm sg}(\widetilde{A})$. It sends a complex $(M^\bullet, \sigma^\bullet)$ to $(X\otimes_B Y\otimes_A M^\bullet, {\rm Id}_{X\otimes_B Y}\otimes\sigma^\bullet)$.
Hence, the composite is isomorphic to the endofunctor $(X\otimes_B Y)\otimes_A-\colon \mathbf{D}_{\rm sg}(\widetilde{A})\rightarrow \mathbf{D}_{\rm sg}(\widetilde{A})$. By Proposition \ref{prop:1}, this is an equivalence and isomorphic to $[-1]\circ S_A$. Here,
we use implicitly that the right $A$-module $X\otimes_B Y$ is projective.

Similarly, the composite functor $\mathbf{D}_{\rm sg}(\widetilde{B}) \stackrel{X\otimes_B-}\longrightarrow \mathbf{D}_{\rm sg}(\widetilde{A}) \stackrel{Y\otimes_A-}\longrightarrow \mathbf{D}_{\rm sg}(\widetilde{B})$ is an equivalence and isomorphic to
$[-1]\circ S_B$. From this, we infer that $Y\otimes_A-$ is a triangle equivalence, whose quasi-inverse is as stated. Here, we use the fact that $S_A^2={\rm Id}_{\mathbf{D}_{\rm sg}(\widetilde{A})}$.
\end{proof}

We have drawn a consequence of Theorem \ref{thm:main} in the introduction. Here, we discuss another one. For this, we
assume now that $k$ is a field, and we consider finite dimensional algebras over $k$. Recall that for an idempotent $e$
in an algebra $A$, we have the corresponding projective left $A$-module $Ae$ and right $A$-module $eA$. Moreover, we have a natural isomorphism $eA\otimes_A Af\simeq eAf$ for any idempotents $e$ and $f$ in $A$. The corresponding $A$-$A$-bimodule $Af\otimes_k eA$ is written as $Af\otimes eA$.

We have the following result, which is based on \cite[Example 3.11]{Ch11}. Recall that a $k$-linear category is \emph{Hom-finite}
provided that all its Hom spaces are finite dimensional. Otherwise, it is \emph{Hom-infinite}. We observe that the singularity category of an algebra is naturally $k$-linear.

\begin{cor}\label{cor:idem}
Let $A$ be a finite dimensional algebra over a field $k$ which has finite global dimension. Let $e$ and $f$ be two idempotents in $A$. Then there is a triangle equivalence $$\mathbf{D}_{\rm sg}(A\oplus (Af\otimes eA))\simeq \mathbf{D}_{\rm sg}(k\oplus eAf).$$
 Consequently, the following statements hold.
\begin{enumerate}
\item The trivial extension algebra $A\oplus (Af\otimes eA)$ has finite global dimension if and only if $eAf=0$.
\item The category $\mathbf{D}_{\rm sg}(A\oplus (Af\otimes eA))$ is nontrivial and Hom-finite if and only if ${\rm dim} \; eAf=1$. In this case, we have a triangle equivalence
    $$\mathbf{D}_{\rm sg}(A\oplus (Af\otimes eA))\simeq k{\rm \mbox{-}mod}.$$
\item The category $\mathbf{D}_{\rm sg}(A\oplus (Af\otimes eA))$ is Hom-infinite if and only if ${\rm dim}\; eAf\geq 2$. In this case, for any non-perfect complexes $X^\bullet$ and $Y^\bullet$ we have
    $${\rm dim}\; {\rm Hom}_{\mathbf{D}_{\rm sg}(A\oplus (Af\otimes eA))}(X^\bullet, Y^\bullet)=\infty.$$
    \end{enumerate}
\end{cor}

For the triangle equivalence in (2), we recall that any semisimple abelian category  has a canonical triangulated
structure such that the shift functor is the identity functor; see \cite[Lemma 3.4]{Ch11}.

\begin{proof}
The first statement follows directly from Theorem \ref{thm:main}. For the consequences, take a basis ${x_1, x_2,\cdots, x_n}$
for $eAf$. Then we have an isomorphism $k\oplus eAf \simeq B_n:=k\langle x_1, x_2, \cdots, x_n\rangle/{(x_1, x_2, \cdots, x_n)^2}$ of algebras. Here, $k\langle x_1, x_2, \cdots, x_n\rangle$ denotes the free algebra. The singularity category $\mathbf{D}_{\rm sg}(B_n)$ is completely described in \cite[Example 3.11]{Ch11}.

  The statement (1) follows from that fact that $B_n$ has finite global dimension if and only if $n=0$; (2) follows from the fact that $\mathbf{D}_{\rm sg}(B_n)$ is nontrivial and Hom-finite if and only if $n=1$; moreover, in this case, we have a triangle equivalence  $\mathbf{D}_{\rm sg}(B_n)\simeq k\mbox{-mod}$. This equivalence might also be deduced from \cite[Theorem 2.1]{Ric}. For the second half of (3), we refer to \cite[Corollary 2.10(4)]{Smi11}.
\end{proof}

We close the paper with a concrete example. The above corollary applies to determine  the singularity category of some non-Gorenstein algebras. Recall that a finite dimensional  algebra $A$ is \emph{Gorenstein} provided that the regular module $A$ has finite injective dimension on both sides. The singularity category of Gorenstein algebras is described in terms of maximal Cohen-Macaulay modules; see \cite{Buc} and \cite{Hap91}. However, not much seems to be known about the singularity category of  non-Gorenstein algebras.

\begin{exm}
{\rm Let $A$ be a finite dimensional algebra given by the following quiver with relations $\{\alpha\beta, \beta'\alpha\beta'\}$. We write the concatenation of paths from the right to the left.
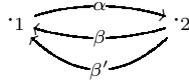
\SelectTips{eu}{10}
\[\xymatrix@!=7pt{
\cdot_1 \ar@/^/[rr]|{\alpha} & &\ar@/^/[ll]|{\beta} \ar@/^1.5pc/[ll]|{\beta'}  \cdot_2 }\]
We observe that ${\rm dim}\; A=11$. Denote by $e_i$ the idempotent corresponding to the vertex $i$, $i=1,2$. Observe that
$1_A=e_1+e_2$.  Denote by $A'$ the subalgebra of $A$ generated by $e_1, e_2, \alpha$ and $\beta$. Then the algebra $A'$ has global dimension two; moreover, $e_2A'e_1=k\alpha$ is of dimension $1$. There is an obvious isomorphism $A\simeq A'\oplus (A'e_1\otimes e_2A')$ of algebras, which sends $\beta'$ to $e_1\otimes e_2$. By Corollary \ref{cor:idem}(2), we have a triangle equivalence
$$\mathbf{D}_{\rm sg}(A)\simeq k\mbox{-mod}.$$
We observe that the algebra $A$ is non-Gorenstein, since the injective hull $I_1$ of the simple module $S_1$ corresponding to the vertex $1$ has infinite projective dimension.}
\end{exm}

\bibliography{}

\begin{thebibliography}{9999}






\bibitem{ALPS} {\sc G. Abrams, A. Louly, E. Pardo and  C. Smith}, {\em
Flow invariants in the classification of Leavitt path algebras}, J. Algebra {\bf 333} (2011), 202--231.







\bibitem{ARS}{\sc M. Auslander, I. Reiten and S.O. Smal{\o},}
Representation Theory of Artin Algebras, Cambridge Studies in Adv.
Math. {\bf 36}, Cambridge Univ. Press, Cambridge, 1995.



%

\bibitem{Bel2000} {\sc A. Beligiannis,} {\em The homological theory of contravariantly finite
subcategories: Auslander-Buchweitz contexts, Gorenstein categories and (co-)stabilization,}
 Comm. Algebra {\bf 28}(10) (2000), 4547--4596.

%
%
%



\bibitem{Buc} {\sc R.O. Buchweitz,} Maximal Cohen-Macaulay Modules
and Tate Cohomology over Gorenstein Rings, Unpublished Manuscript,
1987.


%


\bibitem{Ch09} {\sc X.W. Chen}, {\em Singularity categories, Schur functors and
triangular matrix rings}, Algebr. Represent. Theor. {\bf 12} (2009),
181--191.

\bibitem{Ch11} {\sc X.W. Chen}, {\em The singularity category of an algebra with radical square zero},
arXiv: 1104.4006.

\bibitem{Ch11'} {\sc X.W. Chen}, {\em Singular equivalences induced by homological epimorphisms},
arXiv: 1107.5922.



%
%


%

%
%
%
%
%
%
%
%
%





\bibitem{FGR} {\sc  R.M. Fossum, P.A. Griffith and I. Reiten},  
Trivial Extensions of Abelian Categories, Lecture Notes in Math. {\bf 456}, Springer, Berlin, 1975.


%



\bibitem{Hap91} {\sc D. Happel,} {\em On Gorenstein algebras}, In:
Progress in Math. {\bf 95}, Birkh\"{a}user Verlag, Basel, 1991,
389--404.



\bibitem{Har66} {\sc R. Hartshorne}, Duality and Residue, Lecture
Notes in Math. {\bf 20}, Springer, Berlin, 1966.




\bibitem{KV} {\sc B. Keller and D. Vossieck}, {\em Sous les cat\'{e}gories d\'{e}riv\'{e}es,}
C.R. Acad. Sci. Paris, t. {\bf 305} S\'{e}rie I  (1987) 225--228.


\bibitem{Kra} {\sc H. Krause}, {\em The stable derived category of a noetherian scheme, } Compositio Math. {\bf 141} (2005), 1128--1162.

%
%
%
%
%

\bibitem{Or04} {\sc D. Orlov}, {\em Triangulated categories of singularities and D-branes in Landau-Ginzburg
models}, Trudy Steklov Math. Institute {\bf 204} (2004), 240--262.



%
%


\bibitem{Ric} {\sc J. Rickard}, {\em Derived categories and stable
equivalence,} J. Pure Appl. Algebra {\bf 61} (1989), 303--317.



\bibitem{Smi11} {\sc S.P. Smith}, {\em The non-commutative scheme having a free algebra as a homogeneous coordinate ring},
arXiv:1104.3822.


\bibitem{Smi} {\sc S.P. Smith}, {\em Category equivalences involving graded modules over path algebras of quivers},
arXiv: 1107.3511.

\bibitem{Smi'} {\sc S.P. Smith}, {\em Shift equivalence and a category equivalence involving graded modules over path algebras of quivers}, arXiv: 1108.4994.


%



%
%

\bibitem{Ver77} {\sc J.L. Verdier}, {\em Categories derive\'{e}es},
in SGA 4 1/2, Lecture Notes in Math. {\bf 569}, Springer, Berlin,
1977

\bibitem{Will} {\sc R.F. Williams}, {\em Classification of subshifts of finite type}, Ann. Math. {\bf 98} (1973), 120--153;
{\em erratum}, Ann. Math. {\bf 99} (1974), 380--381.

\end{thebibliography}

\vskip 10pt

 {\footnotesize \noindent Xiao-Wu Chen, Department of
Mathematics, University of Science and Technology of
China, Hefei 230026, Anhui, PR China \\
Wu Wen-Tsun Key Laboratory of Mathematics, USTC, Chinese Academy of Sciences, Hefei 230026, Anhui, PR China.\\
URL: http://mail.ustc.edu.cn/$^\sim$xwchen}

\end{document}